\RequirePackage[l2tabu, orthodox]{nag}
\documentclass[a4paper,final,11pt,leqno
]{article}

\newcommand{\mytitle}{Double-normal pairs in space}

\usepackage{pgf,tikz,bm,enumerate}
\usetikzlibrary{arrows}

\usepackage[final]{microtype}
\makeatletter
\def\MT@register@subst@font{\MT@exp@one@n\MT@in@clist\font@name\MT@font@list
   \ifMT@inlist@\else\xdef\MT@font@list{\MT@font@list\font@name,}\fi}
\makeatother 

\usepackage[small]{titlesec}

\usepackage[notref,notcite]{showkeys}

\usepackage[obeyDraft,backgroundcolor=yellow,linecolor=red,shadow
]{todonotes}

\usepackage{datetime}
\longdate

\usepackage[ruled,vlined]{algorithm2e}

\usepackage{mparhack,fixltx2e}

\usepackage{amsmath, amssymb, amsfonts, amsthm, mathtools}
\usepackage[all,warning]{onlyamsmath}
\usepackage{fixmath}
\mathtoolsset{centercolon}

\usepackage[latin1]{inputenc}
\usepackage[T1]{fontenc}

\usepackage[draft=false,pdftex]{hyperref}
\hypersetup{
    bookmarks=true,         
    unicode=false,          
    pdftoolbar=true,        
    pdfmenubar=true,        
    pdffitwindow=false,     
    pdfstartview={FitH},    
    pdftitle={\mytitle},    
    pdfauthor={Konrad J. Swanepoel},     
    pdfnewwindow=false,      
    colorlinks=true,       
    linkcolor=black,          
    citecolor=black,        
    filecolor=black,      
    urlcolor=black           
}

\theoremstyle{plain}
\newtheorem{theorem}{Theorem}
\newtheorem{lemma}[theorem]{Lemma}
\newtheorem{corollary}[theorem]{Corollary}
\newtheorem{proposition}[theorem]{Proposition}

\theoremstyle{definition}

\usepackage{xspace}

\newcommand{\Prob}[1]{\operatorname{Pr}\left[#1\right]}
\newcommand{\BigProb}[1]{\operatorname{Pr}\Bigl[#1\Bigr]}
\newcommand{\Expec}[1]{\operatorname{E}\left[#1\right]}

\newcommand{\setbuilder}[2]{\left\{#1\;\colon\,#2\right\}}
\newcommand{\set}[1]{\left\{#1\right\}}

\newcommand{\epsi}{\varepsilon}
\newcommand{\fhi}{\varphi}
\newcommand{\norm}[1]{\lVert#1\rVert}

\newcommand{\ipr}[2]{\langle #1, #2 \rangle}

\newcommand{\abs}[1]{\left\lvert#1\right\rvert}

\newcommand{\myangle}{\angle}
\newcommand{\length}[1]{\lvert#1\rvert}
\newcommand{\card}[1]{\left\lvert#1\right\rvert}

\newcommand{\numbersystem}[1]{\mathbb{#1}}

\newcommand{\bN}{\numbersystem{N}}

\newcommand{\bR}{\numbersystem{R}}

\DeclareMathOperator{\conv}{conv}

\DeclareMathOperator{\diam}{diam}

\DeclareMathOperator{\bd}{bd}

\DeclareMathOperator{\lin}{lin}

\newcommand{\vect}[1]{\bm{#1}}
\newcommand{\va}{\vect{a}}
\newcommand{\vb}{\vect{b}}
\newcommand{\vc}{\vect{c}}
\newcommand{\vd}{\vect{d}}

\newcommand{\vo}{\vect{o}}
\newcommand{\vp}{\vect{p}}
\newcommand{\vq}{\vect{q}}

\newcommand{\vu}{\vect{u}}
\newcommand{\vv}{\vect{v}}

\newcommand{\vx}{\vect{x}}
\newcommand{\vy}{\vect{y}}
\newcommand{\vz}{\vect{z}}
\newcommand{\Chi}{\vect{\chi}}

\newcommand{\define}[1]{\emph{#1}}

\title{\mytitle}
\author{J\'anos Pach\thanks{Research partially supported by Swiss National Science Foundation Grants 200021-137574 and 200020-144531, by Hungarian Science Foundation Grant OTKA NN 102029 under the EuroGIGA programs ComPoSe and GraDR, and by NSF grant CCF-08-30272.}\\
EPFL Lausanne and \\
R\'enyi Institute, Budapest\\
\href{mailto:pach@cims.nyu.edu}{\texttt{pach@cims.nyu.edu}}
\and 
Konrad J.\ Swanepoel\\
Department of Mathematics,\\ London School of Economics and Political Science,\\ Houghton Street, London WC2A 2AE, United Kingdom\\
\href{mailto:k.swanepoel@lse.ac.uk}{\texttt{k.swanepoel@lse.ac.uk}}
}
\date{}

\begin{document}
\maketitle

\begin{abstract}
A \define{double-normal pair} of a finite set $S$ of points from $\bR^d$ is a pair of points $\set{\vp,\vq}$ from $S$ such that $S$ lies in the closed strip bounded by the hyperplanes through $\vp$ and $\vq$ perpendicular to~$\vp\vq$.
A double-normal pair $\vp\vq$ is \define{strict} if $S\setminus\set{\vp,\vq}$ lies in the open strip.
The problem of estimating the maximum number $N_d(n)$ of double-normal 
pairs in a set of $n$ points in $\bR^d$, was initiated by Martini and 
Soltan (2006).

It was shown in a companion paper that in the plane, this maximum is $3\lfloor n/2\rfloor$, for every $n>2$.
For $d\geq 3$, it follows from the Erd\H{o}s-Stone theorem 
 in extremal graph theory that $N_d(n)=\frac12(1-1/k)n^2 + o(n^2)$ for a suitable 
positive integer $k=k(d)$. 
Here we prove that $k(3)=2$ and, in general,
$\lceil d/2\rceil \leq k(d)\leq d-1$. Moreover, asymptotically
we have $\lim_{n\rightarrow\infty}k(d)/d=1$.
The same bounds
hold for the maximum number of strict double-normal pairs.
\end{abstract}

\section{Introduction}
Let $V$ be a set of $n$ points in $\bR^d$.
A \define{double-normal pair} of $V$ is a pair of points $\set{\vp,\vq}$ in $V$ such that $V$ lies in the closed strip bounded by the hyperplanes $H_{\vp}$ and $H_{\vq}$ through $\vp$ and $\vq$, respectively, that are perpendicular to $\vp\vq$.
A double-normal pair $\vp\vq$ is \define{strict} if $V\setminus\set{\vp,\vq}$ is disjoint from the hyperplanes $H_{\vp}$ and $H_{\vq}$.
Define the \define{double-normal graph} of $V$ as the graph on the vertex set $V$ in which two vertices $p$ and $q$ are joined by an edge if and only if $\{p,q\}$ is a double-normal pair.
The number of edges of this graph, that is, the number of double-normal pairs induced by $V$ is denoted by $N(V)$.

We define the \define{strict double-normal graph} of $V$ analogously and denote its number of edges by $N'(V)$.

Martini and Soltan \cite[Problems~3 and 4]{martini-soltan-2005} asked for the maximum numbers $N_d(n)$ and $N_d'(n)$ of double-normal pairs and strict double-normal pairs of a set of $n$ points in $\bR^d$:
\[ N_d(n) := \max_{\substack{V\subset \bR^d\\ \card{V}=n}} N(V)\]
and
\[ N_d'(n) := \max_{\substack{V\subset \bR^d\\ \card{V}=n}} N'(V).\]
Clearly, we have $N(V)\geq N'(V)$ and $N_d(n)\geq N_d'(n)$.
It is not difficult to see that $N_2'(n)=n$.
In another paper \cite{PS} we show that $N_2(n)=3\lfloor n/2\rfloor$.
Here we only consider the case $d\geq 3$.
\begin{theorem}\label{thm:3d}
The maximum number of double-normal and strict double-normal pairs in a set of $n$ points in $\bR^3$ satisfy $N_3(n) = n^2/4 + o(n^2)$ and $N_3'(n) = n^2/4 + o(n^2)$.
\end{theorem}
In fact, since the collection of double-normal graphs in Euclidean space is closed under the taking of induced subgraphs, the Erd\H{o}s--Stone~Theorem~\cite{ErdosStone} implies that for each $d\in\bN$, there exist unique $k(d),k'(d)\in\bN$ such that $N_d(n) = \frac{1}{2}(1-\frac{1}{k(d)})n^2 + o(n^2)$ and $N_d'(n) = \frac{1}{2}(1-\frac{1}{k'(d)})n^2 + o(n^2)$.
The number $k(d)$ [resp.\ $k'(d)$] can also be characterised as the largest $k$ such that complete $k$-partite graphs with arbitrarily many points in each class occur as subgraphs of double-normal [resp.\ strictly double-normal] graphs in $\bR^d$.
Theorem~\ref{thm:3d} states that $k(3)=k'(3)=2$ and is a special case of the next theorem.
\begin{theorem}\label{thm:highd}
For each $d$, there exist unique integers $k(d), k'(d)\geq 1$ such that $N_d(n)$, the maximum number of double-normal pairs, and $N_d'(n)$, the maximum number of strict double-normal pairs in a set of $n$ points in $\bR^d$, satisfy \[N_d(n) = \frac{1}{2}\Bigl(1-\frac{1}{k(d)}\Bigr)n^2 + o(n^2)\] and \[N_d'(n) = \frac{1}{2}\Bigl(1-\frac{1}{k'(d)}\Bigr)n^2 + o(n^2).\]
For any $d\geq 3$, we have \[\lceil d/2\rceil\leq k'(d)\leq k(d)\leq d-1.\]
Asymptotically, as $d\to\infty$, we have \[k(d)\geq k'(d)\geq d-O(\log d).\]
\end{theorem}
Although this theorem gives the exact values $k(3)=k'(3)=2$, we do not know whether $k(4)$ or $k'(4)$ equals $2$ or $3$.

Two notions related to double-normal pairs have been studied before.
We define a \define{diameter pair} of $S$ to be a pair of points $\{\vp,\vq\}$ in $S$ such that $\length{\vp\vq}=\diam(S)$.
Note that a diameter pair is also a strictly double-normal pair.
The maximum number of diameter pairs in a set of $n$ points is known for all $d\geq 2$, and in the case of $d\geq 4$, if $n$ is sufficiently large \cite{Erdos60, Grmax, Heppesmax, Strasmax, sw-lenz, Kupavskii}.
We call a pair $\vp\vq$ of a set $S\subset\bR^d$ \define{antipodal} if there exist parallel hyperplanes $H_1$ and $H_2$ through $\vp$ and $\vq$, respectively, such that $S$ lies in the closed strip bounded by the hyperplanes.
The pair is called \define{strictly antipodal} if there exist parallel hyperplanes through $\vp$ and $\vq$ such that $S\setminus\set{\vp,\vq}$ lies in the open strip bounded by the hyperplanes.
Clearly, a (strictly) double-normal pair of a set is also a (strictly) antipodal pair.
The problem of determining the asymptotic behaviour of the maximum number of antipodal or strictly antipodal pairs in a set of $n$ points is open already in $\bR^3$.
For a thorough discussion of antipodal pairs, see the series of papers \cite{MM-I, MM-II, MM-III}.

The paper is structured as follows.
In Section~\ref{sect:3.5}, we collect some geometric lemmas on double-normal pairs.
They are applied in Section~\ref{sect:5} together with a Ramsey-type argument to derive the upper bound of Theorem~\ref{thm:highd} (Theorem~\ref{thm:upperbound}).
Finally, in Section~\ref{sect:6} we show the two lower bounds of Theorem~\ref{thm:highd} (Corollaries~\ref{cor1} and \ref{cor2}).
The asymptotic lower bound follows from a random construction closely related to the construction by Erd\H{o}s and F\"uredi \cite{EF} of strictly antipodal sets of size exponential in the dimension.

We use the following notation.
The inner product of $\vx,\vy\in\bR^d$ is denoted by $\ipr{\vx}{\vy}$, the linear span of $S\subset\bR^d$ by $\lin S$, the convex hull of $S$ by $\conv S$, the diameter of $S$ by $\diam(S)$, the cardinality of a finite set $S$ by $\card{S}$, and the complete $k$-partite graph with $N$ vertices in each class by $K_k(N)$.
An angle with vertex $\vb$ and sides $\vb\va$ and $\vb\vc$ is denoted by $\myangle\va\vb\vc$, which we also use to denote its angular measure.
All angles in this paper have angular measure in the range $(0,\pi)$.
The Euclidean distance between $\vp$ and $\vq$ is denoted $\norm{\vp-\vq}$.

\section{Geometric properties of the double-normal relation}\label{sect:3.5}
Here we collect some elementary geometric properties of double-normals pairs.
They will be used in the next section where we find upper bounds to $k(d)$.

If a unit vector $\vu$ is almost orthogonal to two given unit vectors $\vu_1$ and $\vu_2$, then $\vu$ is still almost orthogonal to any unit vector in the span of $\vu_1$ and $\vu_2$, with an error that becomes worse the closer $\vu_1$ and $\vu_2$ are to each other.
The next lemma quantifies this observation.
\begin{lemma}\label{iprestimate}
Let $\vu,\vu_1,\vu_2$ be unit vectors with $\vu_1\neq\pm\vu_2$, such that for some $\epsi_1,\epsi_2>0$, $\abs{\ipr{\vu}{\vu_1}}\leq\epsi_1$ and $\abs{\ipr{\vu}{\vu_2}}\leq\epsi_2$.
Then for any unit vector $\vv\in\lin\set{\vu_1,\vu_2}$ we have $\abs{\ipr{\vu}{\vv}}<(\epsi_1+\epsi_2)/\sin\theta$, where $\theta\in(0,\pi)$ satisfies $\ipr{\vu_1}{\vu_2}=\cos\theta$.
\end{lemma}
\begin{proof}
Let $\vu'$ be the orthogonal projection of $\vu$ onto the plane $\lin\set{\vu_1,\vu_2}$.
Then the quantity $\ipr{\vu}{\vv}=\ipr{\vu'}{\vv}$ is maximised when $\vv$ is a positive multiple of $\vu'$, and then $\abs{\ipr{\vu}{\vv}}=\norm{\vu'}$.
It follows from the hypotheses that $\vu'$ lies in the parallelogram $P$ symmetric around $\vo$ with sides perpendicular to $\vu_1$ and $\vu_2$, respectively, and with the sides perpendicular to $\vu_i$ at distance $2\epsi_i$, $i=1,2$.
The sides of $P$ form an angle of $\theta$, and their lengths are $2\epsi_1/\sin\theta$ and $2\epsi_2/\sin\theta$.
The maximum value of $\norm{\vu'}$ is attained at a vertex of the parallelogram $P$, that is, $\norm{\vu'}$ is at most half the largest diagonal of $P$.
By the law of cosines, half a diagonal of $P$ has length
\begin{align*}
&\mathrel{\phantom{=}}\sqrt{\frac{\epsi_1^2}{\sin^2\theta}+\frac{\epsi_2^2}{\sin^2\theta}\pm2\frac{\epsi_1\epsi_2}{\sin^2\theta}\cos\theta}\\
&< \sqrt{\frac{\epsi_1^2}{\sin^2\theta}+\frac{\epsi_2^2}{\sin^2\theta}+2\frac{\epsi_1\epsi_2}{\sin^2\theta}}
=\frac{\epsi_1+\epsi_2}{\sin\theta}.\qedhere
\end{align*}
\end{proof}
Suppose that $\vy_1$, $\vy_2$, $\vy_3$ are collinear, with $\vy_2$ between $\vy_1$ and $\vy_3$, and that $\vx\vy_2$ is a double-normal pair in some set that contains $\vx,\vy_1,\vy_2,\vy_3$.
Then, since the segment $\vy_1\vy_3$ has to lie in the half-space through $\vy_2$ with normal $\vy_2\vx$, it follows that $\vy_1\vy_3$ lies in the boundary of this half-space.
That is, $\vx\vy_2\perp\vy_1\vy_2$.
If $\vy_1$, $\vy_2$, $\vy_3$ are close to collinear, then intuitively $\vy_1\vy_2$ will still be close to orthogonal to $\vx\vy_2$.
This is the content of the next lemma.
\begin{lemma}\label{dn1}
Let $\vx,\vy_1,\vy_2,\vy_3$ be different points from $V\subset\bR^d$, with $\vx\vy_2$ a double-normal pair in $V$.
Let $\epsi>0$ and suppose that $\myangle \vy_1\vy_2\vy_3>\pi-\epsi$.
Let $\vu$ be a unit vector parallel to $\vy_1\vy_2$ and $\vv$ a unit vector parallel to $\vx\vy_2$.
Then $\abs{\ipr{\vu}{\vv}}<\epsi$.
\end{lemma}
\begin{proof}
Without loss of generality, $\epsi < \pi/2$.
Note that $\myangle \vx\vy_2\vy_1, \myangle \vx\vy_2\vy_3\leq\pi/2$.
Since also
\[ \pi-\epsi < \myangle \vy_1\vy_2\vy_3\leq \myangle\vy_1\vy_2\vx+\myangle\vx\vy_2\vy_3\leq\myangle\vy_1\vy_2\vx+\pi/2,\]
we obtain
\[ \pi/2-\epsi < \myangle \vy_1\vy_2\vx\leq\pi/2,\]
and it follows that
\[ \abs{\ipr{\vu}{\vv}}=\cos\myangle\vy_1\vy_2\vx<\cos(\pi/2-\epsi)=\sin\epsi<\epsi.\qedhere\]
\end{proof}
Consider the situation where $\vy_1$, $\vy_2$, $\vy_3$ are ``almost'' collinear with $\vy_2$ the ``middle'' point, but now there are two double-normal pairs $\vx_1\vy_2$ and $\vx_2\vy_2$ in a set that contains $\vx_1,\vx_2,\vy_1,\vy_2,\vy_3$.
Then $\vy_1,\vy_2,\vy_3$ all lie inside the wedge $W$ formed by the intersection of the half-spaces $H_1$ and $H_2$ through $\vy_2$ with normals $\vx_1-\vy_2$ and $\vx_2-\vy_2$, respectively.
If $\vy_1$, $\vy_2$, $\vy_3$ are collinear with $\vy_2$ between $\vy_1$ and $\vy_3$, then necessarily $\vy_1,\vy_2,\vy_3$ all lie on the ``ridge'' $\bd H_1\cap \bd H_2$ of the wedge $W$, and $\vy_1\vy_2$ is orthogonal to the plane $\Pi$ through $\vx_1$, $\vx_2$, $\vy_2$.
If $\vy_1$, $\vy_2$, $\vy_3$ are close to collinear, then intuitively $\vy_1\vy_2$ will still be close to orthogonal to $\Pi$.
The next lemma quantifies this intuition.
It is an immediate consequence of Lemmas~\ref{iprestimate} and \ref{dn1}.
\begin{lemma}\label{dn2}
Let $\vx_1,\vx_2,\vy_1,\vy_2,\vy_3$ be different points in $V\subset\bR^d$, with $\vx_1\vy_2$ and $\vx_2\vy_2$ double-normal pairs in $V$.
Let $\epsi>0$.
Suppose that $\myangle \vy_1\vy_2\vy_3>\pi-\epsi$.
Then for any unit vector $\vu$ parallel to the line $\vy_1\vy_2$ and any unit vector $\vv$ parallel to the plane $\vx_1\vx_2\vy_2$ we have $\abs{\ipr{\vu}{\vv}}<2\epsi/\sin\myangle\vx_1\vy_2\vx_2$.
\end{lemma}
If the angle $\myangle\vx_1\vy_2\vx_2$ in the previous lemma is small, then the bound obtained may be too large to be useful.
In the next lemma, we show that we can still obtain a small upper bound if $\norm{\vy_1-\vy_2}$ is much smaller than $\norm{\vx_1-\vx_2}$.
We need four double-normal pairs instead of the two required by Lemma~\ref{dn2}, but we don not need $\vy_3$.
\begin{lemma}\label{dn3}
Let $\vx_i\vy_j$, $i,j=1,2$, be double-normal pairs in a set $V\subset\bR^d$ that contains $\vx_1,\vx_2,\vy_1,\vy_2$.
Let $\vu$ be a unit vector parallel to $\vy_1\vy_2$ and $\vv$ a unit vector parallel to the plane $\vx_1\vx_2\vy_2$.
Then \[\abs{\ipr{\vu}{\vv}}\leq\frac{\sqrt{2}}{\cos^2\myangle\vx_1\vy_2\vx_2}\frac{\norm{\vy_1-\vy_2}}{\norm{\vx_1-\vx_2}}.\]
\end{lemma}
\begin{proof}
Let $\vu:=\norm{\vy_1-\vy_2}^{-1}(\vy_1-\vy_2)$, $\vu_1:=\norm{\vx_1-\vy_2}^{-1}(\vx_1-\vy_2)$ and $\vu_2:=\norm{\vx_1-\vx_2}^{-1}(\vx_1-\vx_2)$.
Then $\ipr{\vu_1}{\vu_2}=\cos\theta$ where $\theta:=\myangle\vx_2\vx_1\vy_2$.

Since the angles $\myangle\vx_1\vy_1\vy_2$, $\myangle\vx_1\vy_2\vy_1$, $\myangle\vx_2\vy_2\vy_1$ are non-obtuse, we obtain
\begin{align}
\ipr{\vx_1-\vy_1}{\vy_2-\vy_1}\geq 0, \label{i1}\\
\ipr{\vx_1-\vy_2}{\vy_1-\vy_2}\geq 0, \label{i2}\\
\intertext{and}
\ipr{\vy_2-\vx_2}{\vy_2-\vy_1}\geq 0. \label{i3}
\end{align}
From \eqref{i1} we obtain $\ipr{\vx_1-\vy_2}{\vy_2-\vy_1}\geq-\norm{\vy_1-\vy_2}^2$, that is, \[\ipr{\vu}{\vu_1}\leq\norm{\vy_2-\vy_1}/\norm{\vx_1-\vy_2}=:\epsi_1.\]
From \eqref{i2}, $\ipr{\vu}{\vu_1}\geq 0$.
Next, add \eqref{i1} and \eqref{i3} to obtain $\ipr{\vx_2-\vx_1}{\vy_2-\vy_1}\leq\norm{\vy_1-\vy_2}^2$, that is,
\[\ipr{\vu}{\vu_2}\leq\norm{\vy_1-\vy_2}/\norm{\vx_1-\vx_2}=:\epsi_2.\]
The analogues of \eqref{i1} and \eqref{i3} with $\vx_1$ and $\vx_2$ interchanged similarly give $-\ipr{\vu}{\vu_2}\leq\epsi_2$.
By Lemma~\ref{iprestimate}, for any unit vector $\vv$ parallel to the plane $\Pi$ through $\vx_1$, $\vx_2$, $\vy_2$, that is, with $\vv\in\lin\set{\vu_1,\vu_2}$, we have
\begin{equation}\label{i4}
\abs{\ipr{\vu}{\vv}} \leq \frac{\epsi_1+\epsi_2}{\sin\theta}.
\end{equation}
By the law of sines in $\triangle\vx_1\vx_2\vy_2$, 
\[ \frac{\epsi_1}{\epsi_2}=\frac{\norm{\vx_1-\vx_2}}{\norm{\vx_1-\vy_2}}=\frac{\sin\alpha}{\sin\fhi},\]
where $\fhi=\myangle\vx_1\vx_2\vy_2$ and $\alpha:=\myangle\vx_1\vy_2\vx_2$.
It follows from \eqref{i4} that
\[\abs{\ipr{\vu}{\vv}} \leq \frac{\epsi_2}{\sin\theta}\left(1+\frac{\sin\alpha}{\sin\fhi}\right).\]
Since $\alpha,\theta,\fhi\leq\pi/2$ and $\alpha+\theta+\fhi=\pi$,
we have \[\sin\theta,\sin\fhi\geq\sin(\pi/2-\alpha)=\cos\alpha,\] hence
\begin{align*}
\abs{\ipr{\vu}{\vv}} &\leq \frac{\epsi_2}{\cos\alpha}\left(1+\frac{\sin\alpha}{\cos\alpha}\right)
= \frac{\epsi_2}{\cos^2\alpha}(\cos\alpha+\sin\alpha)\\
&\leq \frac{\epsi_2}{\cos^2\alpha}\sqrt{2}
= \frac{\sqrt{2}}{\cos^2\alpha}\frac{\norm{\vy_1-\vy_2}}{\norm{\vx_1-\vx_2}}.\qedhere
\end{align*}
\end{proof}

\section{Upper bound on the number of double-normal pairs}\label{sect:5}
Recall that $k(d)$ denotes the largest $k$ such that for each $N\in\bN$, $K_k(N)$ is a subgraph of some double-normal graph in $\bR^d$.
\begin{theorem}\label{thm:upperbound}
For all $d\geq 3$, we have $k(d)\leq d-1$.
\end{theorem}
This theorem is a straightforward consequence of the following technical result.
\begin{proposition}\label{prop:vectors}
There exist a family of $k=k(d)$ not necessarily distinct points $\set{\vp_1,\dots,\vp_k}$ and a family of $k^2$ not necessarily distinct unit vectors $\setbuilder{\vu_{i,j}}{1\leq i,j\leq k}$, all in $\bR^d$, such that the following holds:
\begin{align}
&\text{$\set{\vp_1,\vp_2,\dots,\vp_k}$ has at least two distinct points and no obtuse angles.}\label{p1}\\
&\text{$\set{\vu_{1,1},\vu_{2,2},\dots,\vu_{k,k}}$ is an orthogonal set.}\label{p2}\\
&\text{If $i\neq j$, then $\vu_{i,j}=-\vu_{j,i}$.}\label{p3}\\
&\text{If $\vp_i\neq\vp_j$, then $\vu_{i,j}=\norm{\vp_j-\vp_i}^{-1}(\vp_j-\vp_i)$.}\label{p6}\\
&\text{For any distinct $i,j$, $\vu_{i,i}$ is orthogonal to $\vu_{i,j}$.}\label{p5}\\
&\text{Each $\vu_{i,i}$ is orthogonal to the subspace $\lin\setbuilder{\vp_j-\vp_1}{j=2,\dots,k}$.}\label{p4}\\
&\text{If $\vp_i=\vp_{i'}\neq \vp_j$, then $\vu_{i,i'}$ is orthogonal to $\vu_{i,j}=\vu_{i',j}$.}\label{p7}
\end{align}
\end{proposition}
\begin{proof}
The proof consists of three steps.

\smallskip
\textbf{Step 1.} We will use a geometric Ramsey-type result from \cite{Pach} and the pigeon-hole principle to show that for any $\epsi>0$ there exists $N$ such that for any $K_k(N)$ with classes $V_1,\dots,V_k$ contained in some double-normal graph in $\bR^d$, there exist points $\va_i,\vb_i,\vc_i\in V_i$ ($i=1,\dots,k$) such that 
\begin{align}
&\myangle\va_i\vb_i\vc_i > \pi-\epsi,\quad i=1,\dots,k,\label{q2}\\
&\norm{\va_{i+1}-\vc_{i+1}}\leq\epsi\norm{\va_i-\vc_i},\quad i=1,\dots,k-1,\label{q3}\\
&\norm{\va_i-\vb_i}\geq\frac{1}{2}\norm{\va_i-\vc_i},\quad i=1,\dots,k.\label{q4}
\end{align}

\textbf{Step 2.} We use the results from Section~\ref{sect:3.5} to show that if we set $\vu_{i,i}=\norm{\va_i-\vb_i}^{-1}(\va_i-\vb_i)$ and $\vu_{i,j}=\norm{\vb_j-\vb_i}^{-1}(\vb_j-\vb_i)$, then
\begin{align}
&\abs{\ipr{\vu_{i,i}}{\vu_{i,j}}}<\epsi,\quad i,j=1,\dots,k,\; i\neq j.\label{q6}\\
&\abs{\ipr{\vu_{i,i}}{\vu_{j,j}}}<4\epsi,\quad i,j=1,\dots,k,\; i\neq j\label{q5}
\end{align}

\textbf{Step 3.} The proposition will follow by setting $\epsi=1/n$ and taking subsequences of the sequences $\va_i^{(n)}, \vb_i^{(n)}, \vc_i^{(n)}$, $i=1,\dots,k$, such that $\vb_i^{(n)}$ converges to $\vp_i$, and each $\vu_{i,j}^{(n)}$ converges, as $n\to\infty$.
The details follow.

\smallskip
Let $\epsi>0$ be given. Write $t=\lceil (\epsi\cos\epsi)^{-1}\rceil$.
In \textbf{Step~1}, applying \cite[Theorem~4]{Pach} we first choose a sufficiently large $N$ depending only on $\epsi$ and $d$ such that each class $V_i$ of any $K_k(N)$ contained in a double-normal graph in $\bR^d$ has a subset $V_i'$ of size $2t^{k-1}+1$ such that
for any $\va,\vb, \vc, \vd$ from the same $V_i'$ with $\va\neq\vb$ and $\vc\neq\vd$, the angle between the lines $\va\vb$ and $\vc\vd$ is less than $\epsi$.
We now replace the original $V_i$ by $V_i'$.
If we assume $\epsi<\pi/3$, we obtain a natural linear ordering (more precisely, a betweenness relation) on the points of each $V_i$, by defining for each $\vx,\vy,\vz\in V_i$ that $\vy$ is \define{between} $\vx$ and $\vz$ if $\myangle\vx\vy\vz>\pi-\epsi$.
Then $\norm{\vy-\vx}<\norm{\vz-\vx}$ whenever $\vy$ is between $\vx$ and $\vz$.

\begin{algorithm}[t]
\For{$i=1$ \KwTo $k$}{
\DontPrintSemicolon (\emph{Note that here $\card{V_j}= 2t^{k-i}+1$ for all $j\geq i$})\;
relabel $V_i,\dots,V_k$ such that $\diam(V_i)=\max\setbuilder{\diam(V_j)}{j>i}$\;
\For{$j=i+1$ \KwTo $k$}{
 \DontPrintSemicolon find $V_j'\subseteq V_j$ such that $\card{V_j'}= 2t^{k-i-1}+1$\;
 \PrintSemicolon\quad and $\diam(V_j')\leq \epsi\diam(V_j)$\;
replace $V_j$ by $V_j'$\;
}
}
\caption{Pruning the sets $V_i$}\label{alg1}
\end{algorithm}
Next we run Algorithm~\ref{alg1} on $V_1,\dots,V_k$.
Note that at the start of the outer \textbf{for} loop, $\card{V_j}= 2t^{k-i}+1$ for all $j=i,\dots,k$.
That we can find a $V_j'$ as required inside the inner \textbf{for} loop, is seen as follows.
Write $V_j=\set{\vp_1,\dots,\vp_{2t^{k-i}+1}}$ with the points in their natural order (where $\vp_j$ is between $\vp_i$ and $\vp_k$ if $\myangle\vp_i\vp_j\vp_k>\pi-\epsi$).
Let $\vp_i'$ be the orthogonal projection of $\vp_i$ onto the line $\ell$ through $\vp_1$ and $\vp_{2t^{k-i}+1}$.
Since $\epsi<\pi/2$, the points $\vp_i'$ are in order on $\ell$, and 
\begin{align*}
\norm{\vp_1-\vp_{2t^{k-i}+1}} &= \norm{\vp_1'-\vp_{2t^{k-i}+1}'}\\
&= \sum_{s=1}^t\norm{\vp_{2t^{k-i-1}(s-1)+1}'-\vp_{2t^{k-i-1}s+1}'}\\
& > \cos\epsi\sum_{s=1}^t\norm{\vp_{2t^{k-i-1}(s-1)+1}-\vp_{2t^{k-i-1}s+1}},
\end{align*}
where the last inequality holds, because the angle between $\ell$ and the line through any two $\vp_i$ is less than $\epsi$.
Thus,
\begin{align*}
&\mathrel{\phantom{<}}\frac{1}{t}\sum_{s=1}^t\norm{\vp_{2t^{k-i-1}(s-1)+1}-\vp_{2t^{k-i-1}s+1}}\\
&< \frac{1}{t\cos\epsi}\norm{\vp_1-\vp_{2t^{k-i}+1}}< \epsi\norm{\vp_1-\vp_{2t^{k-i}+1}}.
\end{align*}
It follows that for some $s\in\set{1,\dots,t}$, \[\norm{\vp_{2t^{k-i-1}(s-1)+1}-\vp_{2t^{k-i-1}s+1}} < \epsi\norm{\vp_1-\vp_{2t^{k-i}+1}}.\]
Let $V_j'=\set{\vp_{2t^{k-i-1}(s-1)+1},\dots,\vp_{2t^{k-i-1}s+1}}$.
Then $\card{V_j'}=2t^{k-i-1}+1$ and \[\diam(V_j') < \epsi\norm{\vp_1-\vp_{2t^{k-i}+1}} = \epsi\diam(V_j).\]
When the algorithm is done, we have sets $V_1,\dots,V_k$ such that $\diam(V_{i+1})\geq\epsi\diam(V_i)$ for each $i=1,\dots,k-1$, and $\card{V_i}=2t^{k-i}+1\geq 3$ for each $i=1,\dots,k$.
Let $\va_i\vc_i$ be a diameter of $V_i$ and choose any $\vb_i\in V_i\setminus\set{\va_i,\vc_i}$.
Then \eqref{q2} and \eqref{q3} hold.
To ensure \eqref{q4}, exchange $\va_i$ and $\vc_i$ if necessary such that $\norm{\va_i-\vb_i}\geq\norm{\vc_i-\vb_i}$.
Then \eqref{q4} follows from the triangle inequality.

\smallskip
In \textbf{Step~2} we show \eqref{q6} and \eqref{q5}.
Let $1\leq i,j\leq k$, $i\neq j$.
Without loss of generality, $i<j$.
Then \eqref{q6} follows upon applying Lemma~\ref{dn1} with $\vx=\vb_i$, $\vy_1=\va_j$, $\vy_2=\vb_j$, $\vy_3=\vc_j$.

If $\myangle\va_i\vb_j\vb_i\geq\pi/6$, then by Lemma~\ref{dn2} with $\vx_1=\va_i$, $\vx_2=\vb_i$, $\vy_1=\va_j$, $\vy_2=\vb_j$, $\vy_3=\vc_j$,
\[ \abs{\ipr{\vu_{i,i}}{\vu_{j,j}}} < \frac{2\epsi}{\sin\myangle\va_i\vb_j\vb_i}\leq\frac{2\epsi}{\sin\pi/6}=4\epsi.\]
If $\myangle\va_i\vb_j\vb_i < \pi/6$, then by Lemma~\ref{dn3} with $\vx_1=\va_i$, $\vx_2=\vb_i$, $\vy_1=\va_j$, $\vy_2=\vb_j$,
\begin{align*}
\abs{\ipr{\vu_{i,i}}{\vu_{j,j}}} &< \frac{\sqrt{2}}{\cos^2\myangle\va_i\vb_j\vb_i}\frac{\norm{\va_j-\vb_j}}{\norm{\va_i-\vb_i}}\\
&< \frac{\sqrt{2}}{\cos^2(\pi/6)}\frac{\norm{\va_j-\vc_j}}{\frac{1}{2}\norm{\va_i-\vc_i}} < (8\sqrt{2}/3)\epsi < 4\epsi,
\end{align*}
which shows \eqref{q5}.

\smallskip
In \textbf{Step~3}, we let $n\in\bN$ be arbitrary, set $\epsi=1/n$, and choose $\va_i^{(n)}$, $\vb_i^{(n)}$, $\vc_i^{(n)}$, $i=1,\dots,k$, as in the first stage of the proof.
We may assume, after translating and scaling each $\bigcup_{i=1}^k V_i^{(n)}$ if necessary, that $\set{\vb_1^{(n)},\dots,\vb_k^{(n)}}$ has diameter $1$ and is contained in the unit ball.
Thus, we may pass to subsequences to assume that for each $i$, $\vb_i^{(n)}$ converges to $\vp_i$, say,
\[\vu_{i,i}^{(n)}:=\norm{\va_i^{(n)}-\vb_i^{(n)}}^{-1}(\va_i^{(n)}-\vb_i^{(n)})\]
converges to $\vu_{i,i}$, say, and
\[\vu_{i,j}^{(n)}:=\norm{\vb_j^{(n)}-\vb_i^{(n)}}^{-1}(\vb_j^{(n)}-\vb_i^{(n)})\]
converges to $\vu_{i,j}$, say.
Then $\diam\set{\vp_1,\dots,\vp_k}=1$, and since there are no obtuse angles in $\set{\vb_1^{(n)},\dots,\vb_k^{(n)}}$, there will still be no obtuse angles between distinct elements of $\set{\vp_1,\dots,\vp_k}$.
Thus, \eqref{p1} holds.
Also, \eqref{p2} follows from \eqref{q5}, 
\eqref{p3} from the definition of $\vu_{i,j}^{(n)}$,
\eqref{p6} from the definitions of $\vu_{i,j}^{(n)}$ and $\vp_i$, and \eqref{p5} from \eqref{q6}.
Properties \eqref{p6} and \eqref{p5} immediately imply that $\vu_{i,i}$ is orthogonal to $\vp_i-\vp_j$ for all $j\neq i$.
Since the subspace $\lin\setbuilder{\vp_i-\vp_j}{j\neq i}$ is the same for all $i$, we obtain \eqref{p4}.

Finally, suppose $\vp_i=\vp_{i'}\neq\vp_j$.
Since $\myangle\vb_i^{(n)}\vb_j^{(n)}\vb_{i'}^{(n)}\to 0$ as $n\to\infty$ and $\triangle\vb_i\vb_{i'}\vb_j$ is not obtuse, we obtain that $\myangle\vb_i^{(n)}\vb_{i'}^{(n)}\vb_j^{(n)}\to \pi/2$ and $\myangle\vb_{i'}^{(n)}\vb_i^{(n)}\vb_j^{(n)}\to \pi/2$ as $n\to\infty$, giving $\vu_{i,i'}\perp\vu_{i,j}$.
This shows \eqref{p7}.
\end{proof}
\begin{proof}[Proof of Theorem~\ref{thm:upperbound}]
Let $k=k(d)$.
Consider the points $\vp_1,\dots,\vp_k$ and vectors $\vu_{i,j}$, $1\leq i,j\leq k$ given by Proposition~\ref{prop:vectors}.
There exist distinct $i$ and $j$ such that $\vp_i\neq\vp_j$.
By~\eqref{p2}, the $k$ unit vectors $\vu_{1,1},\dots,\vu_{k,k}$ are pairwise orthogonal.
By~\eqref{p4}, they are also orthogonal to $\vp_i-\vp_j$, which is a multiple of $\vu_{i,j}$ by~\eqref{p6}.
Thus, we have found $k+1$ pairwise orthogonal vectors.
That is, $k(d)+1\leq d$.
\end{proof}

\section{Constructions with many strict double-normal pairs}\label{sect:6}
\begin{theorem}\label{thm:construct}
Let $m\geq 2$.
Suppose that there exist $m$ points $\vp_1,\dots,\vp_m\in\bR^d$ and $m$ unit vectors $\vu_1,\dots,\vu_m\in\bR^d$ such that, for all triples of distinct $i,j,k$, the angle $\myangle\vp_i\vp_j\vp_k$ is acute, and 
\begin{equation}\label{cond}
\ipr{\vu_i}{\vp_i-\vp_j}<\ipr{\vu_i}{\vp_k-\vp_j}<\ipr{\vu_i}{\vp_j-\vp_i}.
\end{equation}
Then, for any $N\in\bN$, there exists a strict double-normal graph in $\bR^{d+m}$ containing a complete $m$-partite $K_m(N)$.
In particular, $k'(d+m)\geq m$.
\end{theorem}
Geometrically, \eqref{cond} means that if we project the points $\vp_1,\dots,\vp_m$ orthogonally onto the line through $\vp_i$ parallel to $\vu_i$, then the projected points are on the ray from $\vp_i$ in the direction of $\vu_i$, and the furthest one is at less than twice the distance from $\vp_i$ than the closest one (other than~$\vp_i$).
\begin{proof}
Identify $\bR^d$ with the first $d$ coordinates of $\bR^{d+m}$, and let $\vv_1,\dots,\vv_m\in\bR^{d+m}$ be pairwise orthogonal unit vectors that are also orthogonal to $\bR^d$.
We will construct countably infinite sets $V_1,\dots,V_m\subset\bR^{d+m}$, with each $V_i$ on a circular arc through $\vp_i$ in the plane $\Pi_i:=\vp_i+\lin\set{\vu_i,\vv_i}$.
Then we will verify that for any distinct $i,j$ and any $\vx\in V_i$ and $\vy\in V_j$, $\vx\vy$ is a strict double-normal pair of $\bigcup_i V_i$.

We will use a small $\epsi>0$ that will depend only on the given points $\vp_1,\dots,\vp_m$ and vectors $\vu_1,\dots,\vu_m$.
As the proof progresses, we will put finitely many constraints on $\epsi$, all depending only on the points $\vp_i$ and vectors $\vu_i$.

Let $\alpha_i=\min_{j\neq i}\ipr{\vu_i}{\vp_j}$ and $\beta_i=\max_{j}\ipr{\vu_i}{\vp_j}$.
By condition~\eqref{cond}, $\ipr{\vu_i}{\vp_i}-\alpha_i<\beta_i-\alpha_i<\alpha_i-\ipr{\vu_i}{\vp_i}$, hence $\ipr{\vu_i}{\vp_i}<\frac12(\beta_i+\ipr{\vu_i}{\vp_i})<\alpha_i$.
We choose $\epsi>0$ small enough so that $\frac12(\beta_i+\epsi+\ipr{\vu_i}{\vp_i})<\alpha_i-\epsi$ for all $i$.
Choose any $r_i\in(\frac12(\beta_i+\epsi+\ipr{\vu_i}{\vp_i}),\alpha_i-\epsi)$, and set $\vc_i=\vp_i+r_i\vu_i$, $\va_i=\vp_i+(\alpha_i-\epsi)\vu_i$, $\vb_i=\vp_i+(\beta_i+\epsi)\vu_i$, $\vq_i=\vp_i+2r_i\vu_i$ (Fig.~\ref{fig2}).
\begin{figure}
\centering
\definecolor{uuuuuu}{rgb}{0.27,0.27,0.27}
\begin{tikzpicture}[scale=1.1,line cap=round,line join=round,>=triangle 45,x=1.0cm,y=1.0cm]
\draw(2.58,1.68) circle (2.7cm);
\draw (4.28,1.67)-- (2.03,4.32);
\draw (-0.37,2.29)-- (2.8,4.97);
\draw [line width=0.4pt] (2.03,4.32)-- (1.85,4.17);
\draw [line width=0.4pt] (1.85,4.17)-- (2,3.99);
\draw [line width=0.4pt] (2,3.99)-- (2.18,4.15);
\draw [line width=0.4pt] (2.18,4.15)-- (2.03,4.32);
\draw (-0.12,1.7)-- (5.28,1.66);
\begin{scope}[shift={(-2,3)}]
\draw[color=black, -latex] (0,0) -- (0.6,0) node[right] {$\vu_i$};
\draw[color=black, -latex] (0,0) -- (0,0.6) node[above] {$\vv_i$};
\draw[color=black] (0.15,0) -- (0.15,0.15) -- (0,0.15);
\end{scope}
\fill [color=black] (2.58,1.68) circle (1.5pt);
\draw[color=black] (2.6,1.4) node {$\vc_i$};
\fill [color=black] (5.28,1.66) circle (1.5pt);
\draw[color=black] (5.6,1.65) node {$\vq_i$};
\fill [color=black] (-0.12,1.7) circle (1.5pt);
\draw[color=black,left] (-0.1,1.7) node {$\vp_i$};
\fill [color=black] (4.28,1.67) circle (1.5pt);
\draw[color=black] (4.25,1.42) node {$\vb_i$};
\fill [color=black] (3.66,1.67) circle (1.5pt);
\draw[color=black] (3.65,1.42) node {$\vz$};
\fill [color=black] (3.1,1.68) circle (1.5pt);
\draw[color=black] (3.1,1.4) node {$\va_i$};
\fill [color=black] (2.03,4.32) circle (1.5pt);
\draw[color=black] (1.9,4.55) node {$\vx_1$};
\fill [color=uuuuuu] (0.06,2.66) circle (1.5pt);
\draw[color=uuuuuu] (-0.15,2.7) node {$\vy$};
\fill [color=black] (-0.1,2.04) circle (1.5pt);
\draw[color=black] (0.2,2.1) node {$\vx_2$};
\draw (4.5,4) node {$C_i$};
\end{tikzpicture}
\caption{Constructing $V_i=\setbuilder{\vx_t}{t\in\bN}$}\label{fig2}
\end{figure}
Denote the circle with centre $\vc_i$ and radius $r_i$ in the plane $\Pi_i$ by $C_i$.
Then $\vp_i\vq_i$ is a diameter of $C_i$ parallel to $\vu_i$, and $\va_i$ and $\vb_i$ are strictly between $\vc_i$ and $\vq_i$.
Choose any $\vx_1\in C_i\setminus\set{\vp_i}$ such that $\myangle \vx_1\vc_i\vp_i$ is acute.
We will now recursively choose $\vx_2,\vx_3,\dots$ on the minor arc $\gamma_i$ of $C_i$ between $\vx_1$ and $\vp_i$ such that for any $\vz$ on the segment $\va_i\vb_i$, the angle $\myangle\vz\vx_t\vx_s$ is acute for all distinct $s,t\in\bN$.
Assume that for some $t\in\bN$ we have already chosen $\vx_1,\dots,\vx_t\in\gamma_i$ with $\vx_{s+1}$ between $\vx_s$ and $\vp_i$ for each $s=1,\dots,t-1$, and such that $\myangle\vz\vx_j\vx_k$ is acute for all $1\leq j,k\leq t$, $j\neq k$, and for all $\vz$ on the segment $\va_i\vb_i$.
Since $\vp_i\vx_t\vq_i$ is a right angle, $\myangle\vp_i\vx_t\vb_i$ is acute, and the line in $\Pi_i$ through $\vx_t$ and perpendicular to $\vb_i\vx_t$ intersects $C_i$ in a point $\vy\in\gamma_i$ between $\vx_t$ and $\vp_i$.
Let $\vx_{t+1}$ be any point on $\gamma_i$ between $\vy$ and $\vp_i$.
Now consider any $\vz$ on the segment $\va_i\vb_i$.
We have to show that $\myangle \vz\vx_{t+1}\vx_s$ and $\vz\vx_s\vx_{t+1}$ are acute for all $s=1,\dots,t$.
This can be simply seen as follows:
\[ \myangle\vz\vx_{t+1}\vx_s \leq \myangle \vz\vx_{t+1}\vx_t \leq \myangle \vc_i\vx_{t+1}\vx_t < \pi/2\]
and
\[ \myangle \vz\vx_{s}\vx_{t+1} \leq\myangle\vz\vx_t\vx_{t+1}\leq\myangle\vb_i\vx_t\vx_{t+1}<\myangle\vb_i\vx_t\vy=\pi/2.\]
Finally, let $V_i=\setbuilder{\vx_t}{t\in\bN}$.
Then $\diam V_i=\norm{\vp_i-\vx_1}$, which can be made arbitrarily small by choosing $\vx_1$ close enough to $\vp_i$.
We can assume that all $\diam(V_i)<\epsi$. 
This finishes the construction.

Let $1\leq i < j \leq m$, $\vx\in V_i$ and $\vy\in V_j$.
We have to show that all $\vz\in\bigcup_i V_i\setminus\set{\vx,\vy}$ are in the open slab bounded by the hyperplanes through $\vx$ and $\vy$ orthogonal to $\vx\vy$.
First consider the case where $\vz\in V_k$, $k\neq i,j$.
Since $\myangle\vp_i\vp_j\vp_k$ and $\myangle\vp_j\vp_i\vp_k$ are acute, $\ipr{\vp_i-\vp_j}{\vp_k-\vp_j}>0$ and $\ipr{\vp_j-\vp_i}{\vp_k-\vp_i}>0$.
Noting that $\norm{\vx-\vp_i},\norm{\vy-\vp_j},\norm{\vz-\vp_k}<\epsi$, it follows that
$\ipr{\vx-\vy}{\vz-\vy}>0$ and $\ipr{\vy-\vx}{\vz-\vx}>0$
if $\epsi$ is sufficiently small, depending only on the given points.
That is, $\vz$ is in the open slab determined by $\vx\vy$.

Next consider the case where $\vz\in V_i\cup V_j$.
Without loss of generality, $\vz\in V_i$.
Then \[\ipr{\vx-\vy}{\vz-\vy}=\ipr{\vx-\vy}{\vz-\vx}+\norm{\vx-\vy}^2 \geq -\epsi\norm{\vx-\vy}+\norm{\vx-\vy}^2 > 0,\] as long as $\epsi<\norm{\vx-\vy}$.
It remains to verify that $\ipr{\vy-\vx}{\vz-\vx}>0$.
Denote the orthogonal projection of a point $\vp\in\bR^{d+m}$ onto the plane $\Pi_i$ by $\vp'$.
Since $V_j\subset\Pi_j\subseteq\bR^d+\lin\set{\vv_j}$, it follows that $\vp_j',\vy'\in\vp_i+\lin\set{\vu_i}$.
In particular, $\vp_j'$ is also the orthogonal projection of $\vp_j$ onto the line $\vp_i+\lin\set{\vu_i}$.
By hypothesis, $\vp_j'=\vp_i+\lambda\vu_i$ for some $\lambda\in[\alpha_i,\beta_i]$.
Since $\norm{\vp_j'-\vy'}\leq\norm{\vp_j-\vy}<\epsi$, it follows that $\vy'=\vp_i+\mu\vu_i$ where $\mu\in[\alpha_i-\epsi,\beta_i+\epsi]$, that is, $\vy'$ is on the segment $\va_i\vb_i$.
By construction, the angle $\myangle\vy'\vx\vz$ is acute, hence $\ipr{\vy-\vx}{\vz-\vx}=\ipr{\vy'-\vx}{\vz-\vx}>0$.
\end{proof}
\begin{corollary}\label{cor1}
$k'(d)\geq \lceil d/2\rceil$.
\end{corollary}
\begin{proof}
Let $m=\lceil d/2\rceil$.
Let $\vp_1,\dots,\vp_m$ be the vertices of a regular simplex in $\bR^{m-1}$ inscribed in the unit sphere.
Then the $\vp_i$ and $\vu_i:=-\vp_i$ satisfy the conditions of Theorem~\ref{thm:construct}.
It follows that $k'(d)\geq k'(2m-1)\geq m$.
\end{proof}
\begin{theorem}\label{thm:exist}
There exist $m=\lfloor\frac14e^{d/20}\rfloor$ distinct points $\vp_1,\dots,\vp_m\in\bR^d$ and unit vectors $\vu_1,\dots,\vu_m\in\bR^d$ such that for all distinct $1\leq i,j,k\leq m$, the angle $\myangle\vp_i\vp_j\vp_k$ is acute, and condition~\eqref{cond} is satisfied.
\end{theorem}
The proof of Theorem~\ref{thm:exist} is probabilistic, and is a modification of an argument of Erd\H{o}s and F\"uredi \cite{EF}.
Write $[d]$ for the set $\{1,2,\dots,d\}$ of all integers from $1$ to $d$.
For any $A\subseteq[d]$, let $\Chi(A)\in\set{0,1}^d$ denote its characteristic vector.
The routine proofs of the following three lemmas are omitted.
\begin{lemma}[{\cite[Lemma~2.3]{EF}}]\label{perp}
Let $A$, $B$, and $C$ be distinct subsets of $[d]$.
Then we have $\myangle\Chi(A)\Chi(C)\Chi(B)\leq\pi/2$, and equality holds iff
$A\cap B\subseteq C\subseteq A\cup B$.
\end{lemma}
\begin{lemma}[\cite{EF}]\label{perpprob}
If $A$, $B$, and $C$ are subsets of $[d]$ chosen independently and uniformly, then we have $\BigProb{A\cap B\subseteq C\subseteq A\cup B} = (3/4)^d$.
\end{lemma}
\begin{lemma}\label{condtriple}
Let $A,B,C\subseteq[d]$ and consider the unit vector
\[ \vu:=(1/\sqrt{d})(\Chi([d])-2\Chi(A)).\]
Then we have $\ipr{\vu}{\Chi(A)}\leq\ipr{\vu}{\Chi(B)}$, with equality if and only if $A=B$.
Also, \[\ipr{\vu}{\Chi(B)-\Chi(C)}\geq\ipr{\vu}{\Chi(C)-\Chi(A)}\] if and only if \[4\card{A\cap C}+\card{B}\geq2\card{A\cap B}+\card{A}+2\card{C}.\]
\end{lemma}
\begin{lemma}\label{condprob}
If $A$, $B$, and $C$ are subsets of $[d]$ chosen independently and uniformly, then we have \[\BigProb{4\card{A\cap C}+\card{B}\geq2\card{A\cap B}+\card{A}+2\card{C}}\leq \left(\frac{65}{72}\right)^d<e^{-d/10}.\]
\end{lemma}
\begin{proof}
Let $X$ be the random variable
\[ X = 4\card{A\cap C}+\card{B}-2\card{A\cap B}-\card{A}-2\card{C} = \sum_{i=1}^d X_i,\]
where $X_i$ is the contribution of the element $i\in[d]$ to $X$, that is,
\[ X_i = \begin{cases}
\phantom{-}1 &\text{if\, $i\in B\setminus(A\cup C)$ or $i\in(A\cap C)\setminus B$,}\\
\phantom{-}0 &\text{if\, $i\in A\cap B\cap C$ or $i\notin A\cup B\cup C$,}\\
-1 &\text{if\, $i\in A\setminus(B\cup C)$ or $i\in(B\cap C)\setminus A$,}\\
-2 &\text{if\, $i\in C\setminus(A\cup B)$ or $i\in(A\cap B)\setminus C$.}
\end{cases}\]
Note that \[\Prob{X_i=1}=\Prob{X_i=0}=\Prob{X_i=-1}=\Prob{X_i=-2}=1/4.\]
We now bound $\Prob{X\geq 0}$ from above.
For any $\lambda\geq 1$,
\begin{align*}
&{}\mathrel{\phantom{=}} \Prob{X\geq 0} = \Prob{\lambda^X\geq 1}\\
&\leq \Expec{\lambda^X}= \prod_{i=1}^d\Expec{\lambda^{X_i}}= \left(\frac{\lambda+1+\lambda^{-1}+\lambda^{-2}}{4}\right)^d,
\end{align*}
where we used Markov's inequality and independence.
Set $\lambda=3/2$, which is close to minimizing the right-hand side.
This gives $\Prob{X\geq 0}\leq(65/72)^d$.
\end{proof}
\begin{proof}[Proof of Theorem~\ref{thm:exist}]
Let $m:=\lfloor (1/4)e^{d/20}\rfloor$.
Choose subsets $A_1,\dots,A_{2m}$ randomly and independently from the set $[d]$.
For $i\in[d]$, define $\vp_i=\Chi(A_i)$ and $\vu_i=(1/\sqrt{d})(\Chi([d])-2\Chi(A_i))$.
Let $i,j,k\in[d]$ be distinct.

Assume that $A_i$, $A_j$, $A_k$ are distinct sets.
Then by Lemma~\ref{perp}, $\myangle\vp_i\vp_k\vp_j$ fails to be acute if and only if
\begin{equation}\label{*}
A_i\cap A_j\subseteq A_k\subseteq A_i\cup A_j,
\end{equation}
and condition~\eqref{cond} is violated if and only if
\begin{equation}\label{+}
\ipr{\vu_i}{\Chi(A_i)-\Chi(A_j)} \geq \ipr{\vu_i}{\Chi(A_k)-\Chi(A_j)}
\end{equation}
or
\begin{equation}\label{++}
\ipr{\vu_i}{\Chi(A_k)-\Chi(A_j)} \geq \ipr{\vu_i}{\Chi(A_j)-\Chi(A_i)}.
\end{equation}
Condition~\eqref{+} is equivalent to $\ipr{\vu_i}{\Chi(A_i)}\geq\ipr{\vu_i}{\Chi(A_k)}$.
This, in turn, is equivalent to $A_i=A_k$, by the first statement of Lemma~\ref{condtriple}, contradicting our assumption that $A_i$, $A_j$, $A_k$ are distinct.
By the second statement of Lemma~\ref{condtriple}, \eqref{++} is equivalent to
\begin{equation}\label{**}
4\card{A_i\cap A_j}+\card{A_k}\geq 2\card{A_i\cap A_k}+\card{A_i}+2\card{A_j}.
\end{equation}
Thus, for distinct points $\vp_i$, $\vp_j$, $\vp_k$, at least one of the conditions~\eqref{*} and \eqref{**} holds if and only if $\myangle\vp_i\vp_k\vp_j$ is a right angle or condition~\eqref{cond} is violated.

Note that if some two of the sets coincide, say $A_i=A_k$, then \eqref{*} also holds.
Let us call a triple of distinct numbers $(i,j,k)$ \emph{bad} if at least one of \eqref{*} and \eqref{**} holds.
It follows that if no triple $(i,j,k)$ is bad, then all points $\vp_i$ are distinct, all angles $\myangle\vp_i\vp_j\vp_k$ are acute, and condition~\eqref{cond} is also satisfied.
We will show that with positive probability, some $m$ of the $A_1,\dots,A_{2m}$ will be without bad triples, which will prove the theorem.

By Lemmas~\ref{perpprob} and \ref{condprob} and the union bound, we obtain that
\[\BigProb{(i,j,k)\text{ is bad}}\leq(3/4)^d+e^{-d/10} < 2e^{-d/10}.\]
By linearity of expectation, the expected number of bad triples is at most
\[ 2m(2m-1)(2m-2)2e^{-d/10} < 16m^3e^{-d/10}.\]
In particular, there exists a choice of subsets $A_1,\dots,A_{2m}\subseteq[d]$ with less than $16m^3e^{-d/10}$ bad triples.
For each bad triple $(i,j,k)$, remove $A_i$ from $\set{A_1,\dots,A_{2m}}$.
We are left with more than $2m - 16m^3e^{-d/10}$ sets without any bad triple.
Since $m\leq(1/4)e^{d/20}$ implies that $2m - 16m^3e^{-d/10}\geq m$, we obtain $m$ points $\vp_i$ with unit vectors $\vu_i$ satisfying the theorem.
\end{proof}
\begin{corollary}\label{cor2}
$k'(d)\geq d - O(\log d)$.
\end{corollary}
\begin{proof}
Let $n$ be the unique integer such that
\[ \lfloor (1/4)e^{n/20}\rfloor + n \leq d < \lfloor(1/4)e^{(n+1)/20}\rfloor + n+1.\]
By Theorems~\ref{thm:exist} and \ref{thm:construct}, $k'(m+n+1)\geq m$ for any $m=2,\dots,\lfloor(1/4)e^{(n+1)/20}\rfloor$.
In particular, if we take $m=d-n-1$, we obtain \[k'(d)\geq d-n-1 > d-20\log(4d) - 1.\qedhere\]
\end{proof}

\subsection*{Acknowledgement}
We thank Endre Makai for a careful reading of the manuscript and for many enlightening comments.

\end{document}